\documentclass[12pt, reqno]{amsart}
\allowdisplaybreaks[1]
\usepackage{amsmath}
\usepackage{amssymb}
\usepackage{amsfonts}
\usepackage{verbatim}
\usepackage{amsthm}
\usepackage[usenames]{color}
\usepackage{hyperref}
\makeindex

 \newtheorem{theorem}{Theorem}[section]
 \newtheorem{corollary}[theorem]{Corollary}
  
 \newtheorem{lemma}[theorem]{Lemma}

\theoremstyle{definition}

\theoremstyle{remark}

\newtheorem{fact*}{Fact}
\newcommand{\til}{\raise.17ex\hbox{$\scriptstyle\mathtt{\sim}$}}

\newcommand{\HH}{\mathcal{H}}
\newcommand{\MH}{\textrm{Mult }\mathcal{H}}

\newcommand\beq{\begin{equation}}

\newcommand\eeq{\end{equation}}

\newcommand{\bbm}{\left[ \begin{matrix}}
\newcommand{\ebm}{\end{matrix} \right]}
\newcommand{\bpm}{\left( \begin{matrix}}
\newcommand{\epm}{\end{matrix} \right)}
\numberwithin{equation}{section}

\newlength{\Mheight}
\newlength{\cwidth}

\newcommand{\dfn}[1]{{\bf #1}\index{#1}}

\title[The Random Column-Row Property]{Committee spaces and the random column-row property}
\author{
J. E. Pascoe 
}
\date{\today}

\begin{document}

\begin{abstract}
A committee space is a Hilbert space of power series, perhaps in several or noncommuting variables, such that $\|z^\alpha\|\|z^\beta\| \geq \|z^{\alpha+\beta}\|.$
Such a space satisfies the true column-row property when ever the map transposing a column multiplier to a row multiplier is contractive. We describe a model for random multipliers and show that
such random multipliers satisfy the true column-row property.
We also show that the column-row property holds asymptotically for large random Nevanlinna-Pick interpolation problems where the nodes are chosen identically and independently.
These results suggest that finding a violation of the true column-row property for the Drury-Arveson space via na\"ive random search is unlikely.
\end{abstract}
\maketitle

%\tableofcontents

\section{Introduction}
Let $\HH$ be a Hilbert space of formal power series such that polynomials are dense. (Either commutative or not, and perhaps in several or infinitely many variables.)
We say $\HH$ is a \dfn{committee space} whenever 
\begin{enumerate}
	\item $\langle z^\alpha, z^\beta \rangle = 0$ if $\alpha \neq \beta,$
	\item \dfn{Committee inequality:} $\|z^\alpha\|\|z^\beta\| \geq \|z^{\alpha+\beta}\|.$
\end{enumerate}
One can verify that the following spaces are committee spaces:
\begin{itemize}
		\item Hardy space. The space of power series in one variable such that each monomial has norm one.
		\item The Drury-Arveson space. The Hilbert space of commuting power series in $d$ variables such that $$\langle z^\alpha, z^\alpha\rangle = \frac{1}{{{|\alpha|}\choose{\alpha}}},$$
		where ${|\alpha|}\choose{\alpha}$ is the multi-nomial coefficient. That is, if $\alpha =(\alpha_1,\ldots \alpha_d),$
			$${{|\alpha|}\choose{\alpha}} = \frac{(\alpha_1+\ldots +\alpha_d)!}{\alpha_1!\ldots\alpha_d!},$$
		which is number of ways to divide $|\alpha|$ people into committees of size $\alpha_i.$
		The committee inequality follows from the fact that 
			$${{|\alpha+\beta|}\choose{\alpha+\beta}} \geq {{|\alpha|}\choose{\alpha}}{{|\beta|}\choose{\beta}}$$
		which in turn corresponds to the fact there are less ways to divide people into committees when constraints are placed on the composition of those committees.
		\item Fock space. The space of noncommutative power series such that each monomial has norm one.
		\item Dirichlet space. The space of power series in one variable such that the $\langle z^n, z^n\rangle = n+1.$
\end{itemize}

We denote the monomial shifts on $\HH$ by $M_{z^\gamma}.$
Note that $M_{z^\gamma}M_{z^\gamma}^*$ and $M_{z^\gamma}^*M_{z^\gamma}$ are diagonal
and with respect to the basis $z^\alpha,$ and therefore $\|M_{z^\gamma}\| = \sup_\alpha \frac{\|z^{\gamma+\alpha}\|}{\|z^\alpha\|}$
is bounded by the committee inequality. The multipliers $M_f$ of $\HH$ are denoted as $\MH.$

We say $\HH$ satisfies the \dfn{true column-row property} if $\|\sum_i M_{f_i}M_{f_i}^*\|\leq \|\sum_i M_{f_i}^*M_{f_i}\|$ for any sequence of multipliers.
We say $\HH$ satisfies the \dfn{column-row property} if there is a constant $C>0$ such that $\|\sum_i M_{f_i}M_{f_i}^*\|\leq C\|\sum_i M_{f_i}^*M_{f_i}\|$ for any sequence of multipliers.
The column-row property  is important in interpolation theory \cite{Trent2004, AHMR, AHMR2, JuryMartin}. So far as the author knows, there are no known commutative complete Nevanlinna-Pick spaces
for which the true column-row property fails, and, in general, complete Nevanlinna-Pick spaces are committee spaces. The column row property fails for the Fock space in two or more variables
\cite{AJKP}.

The goal of this manuscript is to understand when a random sequence of multipliers $M_{f_1}, M_{f_2}, \ldots$ 
satifies the column-row property. That is, when is
	$$\|\lim_{n\rightarrow \infty} \frac{1}{n}\sum^n_{i=1} M_{f_i}M_{f_i}^*\|\leq \|\lim_{n\rightarrow \infty} \frac{1}{n}\sum^n_{i=1} M_{f_i}^*M_{f_i}\|.$$
Here a normalization is taken to guarantee convergence.
We note a slight subtlety here, the limits must be evaluated in the weak operator topology as opposed to in norm.
If our method of sampling secretly sampled from a space with finite dimension, all limits would reduce to  \emph{bona fide}
norm limits. 

\section{A model for sampling random multipliers}
	Let $v_\gamma$ be a sequence of random variables indexed by multi-indices $\gamma$ such that
	\begin{enumerate}
		\item $E(v_\gamma) = 0,$
		\item $E(|v_\gamma|^2) = w_\gamma < \infty,$
		\item $E(\overline{v_\gamma} v_{\gamma'})=0$ if $\gamma \neq \gamma',$
		\item There is constant $C>0$ such that the function $f_v = \sum v_\gamma z^\gamma$ satisfies $\|M_{f_v}\|\leq C$ almost surely.
	\end{enumerate}
	We call such a sequence $v_\gamma$ a \dfn{random multiplier model.}
	We will compute
		$$R_v = E(M_{f_v}M_{f_v}^*), C_v = E(M_{f_v}^*M_{f_v}).$$
	Note that, in the weak operator topology, almost surely,
		$$R_v = \lim_{n\rightarrow \infty} \frac{1}{n}\sum^n_{i=1} M_{f_i}M_{f_i}^*,$$
		$$C_v = \lim_{n\rightarrow \infty} \frac{1}{n}\sum^n_{i=1} M_{f_i}^*M_{f_i}$$
	where $f_i$ is a sequence of functions sampled from the random multiplier model.
\section{The row column norm holds for random multipliers}
	We will now prove  the following theorem.
	\begin{theorem}\label{mainresult}
	Let $\HH$ be a committee space. Let $v_\gamma$ be a random multiplier model.
	Then, 
		$$\|R_v\|\leq \|C_v\|.$$
	\end{theorem}
	\begin{proof}
	\begin{lemma}
		Let $\alpha, \gamma$ be multi-indices.
		If there is a $\beta$ such that $\alpha = \gamma + \beta,$
		then,
		$$\|M_{z^\gamma}^*z^\alpha\| = \frac{\|z^\alpha\|^2}{\|z^\beta\|}.$$
		Otherwise,
			$$\|M_{z^\gamma}^*z^\alpha\| =0.$$
	\end{lemma}
	\begin{proof}
		Note that if $$\langle M_{z^\gamma}^*z^\alpha, z^\beta \rangle = \langle z^\alpha,  M_{z^\gamma}z^\beta \rangle = \langle z^\alpha,  z^{\gamma+\beta} \rangle$$
		is to be non-zero, then $\alpha = \gamma + \beta.$ Moreover, $\beta$ is unique.
		Now,
		\begin{align*}
			\|M_{z^\gamma}^*z^\alpha\| 	&= \langle M_{z^\gamma}^*z^\alpha, \frac{z^\beta}{\|z^\beta\|} \rangle \\
							&= \langle z^\alpha, \frac{z^\alpha}{\|z^\beta\|} \rangle \\
							&=  \frac{\|z^\alpha\|^2}{\|z^\beta\|} .
		\end{align*}
	\end{proof}
	\begin{lemma}
		$$\|R_v\| \leq \sup_{\alpha} \sum_{\gamma+\beta=\alpha} \|z^\gamma\|^2w_\gamma.$$
	\end{lemma}
	\begin{proof}
		Note,
		\begin{align*}		
			R_v 	&= E(M_{f_v}M_{f_v}^*)\\
				&= E(\sum_{\gamma,\gamma'} v_\gamma \overline{v_{\gamma'}}M_{z^{\gamma}}M_{z^{\gamma'}}^*) \\
				&= \sum_{\gamma} w_\gamma M_{z^{\gamma}}M_{z^{\gamma}}^*. 
		\end{align*}
		So, $R_v$ must be a diagonal operator with respect to the monomial basis.
		Therefore, we may compute the norm as follows,
			\begin{align*}		
			\|R_v\| &= \sup_\alpha \frac{\langle R_v z^\alpha, z^\alpha\rangle}{\|z^\alpha\|^2}\\
				&= \sup_\alpha \frac{\langle \sum_{\gamma} w_\gamma M_{z^{\gamma}}M_{z^{\gamma}}^*z^\alpha, z^\alpha\rangle}{\|z^\alpha\|^2} \\
				&= \sup_\alpha \sum_{\gamma+\beta=\alpha} \frac{ \|z^{\alpha}\|^4w_\gamma}{\|z^\beta\|^2\|z^\alpha\|^2} \\
				&= \sup_\alpha \sum_{\gamma+\beta=\alpha} \frac{ \|z^{\beta+\gamma}\|^2w_\gamma}{\|z^\beta\|^2} \\
				&\leq \sum_{\gamma+\beta = \alpha} \|z^{\gamma}\|^2w_\gamma,
		\end{align*}
		where the last line holds by the committee inequality.
	\end{proof}
	\begin{lemma}\label{collemma}
		$$\|C_v\| = \sum_{\gamma} \|z^\gamma\|^2w_\gamma.$$
	\end{lemma}
	\begin{proof}
		Note,
		\begin{align*}		
			C_v 	&= E(M_{f_v}^*M_{f_v})\\
				&= E(\sum_{\gamma,\gamma'} \overline{v_\gamma} v_{\gamma'}M_{z^{\gamma}}^*M_{z^{\gamma'}}) \\
				&= \sum_{\gamma} w_\gamma M_{z^{\gamma}}^*M_{z^{\gamma}}. 
		\end{align*}
		So, $C_v$ must be a diagonal operator with respect to the monomial basis.
		Therefore, we may compute the norm as follows,
		\begin{align*}		
			\|C_v\| &= \sup_\alpha \frac{\langle C_v z^\alpha, z^\alpha\rangle}{\|z^\alpha\|^2}\\
				&= \sup_\alpha \frac{\langle \sum_{\gamma} w_\gamma M_{z^{\gamma}}^*M_{z^{\gamma}}z^\alpha, z^\alpha\rangle}{\|z^\alpha\|^2} \\
				&= \sup_\alpha \sum_{\gamma} \frac{ \|z^{\gamma+\alpha}\|^2w_\gamma}{\|z^\alpha\|^2} \\
				&= \sum_{\gamma} \|z^{\gamma}\|^2w_\gamma,
		\end{align*}
		where the last equality follows from the committee inequality, as that implies the maximum is attained when $\alpha$ is the trivial multi-index.
	\end{proof}
	Now, we see that 
		$$\|R_v\|\leq \sup_{\alpha} \sum_{\gamma+\beta=\alpha} \|z^\gamma\|w_\gamma \leq
		\sum_{\gamma} \|z^\gamma\|w_\gamma \leq \|C_v\|,$$
	which proves Theorem \ref{mainresult}.
	\end{proof}

	Note that, formally, we could have assumed there was constant $C>0$ such that the function $f_v = \sum v_\gamma z^\gamma$ satisfies $\|f_v\|_{\HH}\leq C$ almost surely, and, algebraically, everything would have still worked.
	This is reminiscent of the theorem of Cochran-Shapiro-Ullrich \cite{CSU}, that, given a function in the Dirichlet space, randomly multiplying each coefficient by $\pm 1$ gives a multiplier of the Dirichlet space.
	\subsection{The truncated shift case}
	We note that the above calculations hold true, with mild but insightful cosmetic differences,
	for the restriction of the shifts to monomials of degree less than $n.$
	We see that, in fact,
		$$\|R^{(n)}_v\|\leq \max_{|\alpha|\leq n} \sum_{\gamma+\beta=\alpha} \|z^\gamma\|w_\gamma \leq
		\sum_{|\gamma|\leq n} \|z^\gamma\|w_\gamma \leq \|C^{(n)}_v\|.$$
	When we are working in more than one variable, one can see that, given a multi-index $\alpha$ of degree less than or equal to $n,$ the number of $\gamma$ such that 
	$\gamma+\beta=\alpha$ is very small compared to the set of all multi-indices $\gamma$ of degree less than $n.$ Thus, heuristically, one expects that $\|R_v\|$ will be much smaller than $\|C_v\|,$
	although detailed estimates will depend intricately on the parameters $w_\gamma.$ We interpret this as explanation of the fact that numerical experiments to test the column row property have not produced counterexamples.

	\section{The random basis lemma}
		We say a $\HH$-valued random variable $h$ is a \dfn{random vector} if $P(h \perp g)<1$ for all $g\in \HH.$

		The goal of this section is to prove the following lemma, which we will need for technical reasons later. The content is essentially
		that an infinite sequence of random vectors is a (perhaps overdetermined) basis.
		\begin{lemma}\label{mainlemma}
			If $h_1, h_2, \ldots$ is a sequence independent identically distributed of random vectors,
			then, almost surely, the closed span of the $h_i$ is equal to $\HH.$ 
		\end{lemma}
		\begin{proof}
		First we will need a lemma.
		\begin{lemma}\label{Aconstruction}
			Let $h$ be a random vector.
			There is a countable subset $A$ of $\HH$ such that the closed span of the elements of $A$ is equal to $\HH$
			and for every point $a\in A,$ $P(h\in U)>0$ for any neighborhood $U$ of $a.$
		\end{lemma}
		\begin{proof}
			For any subset $A$ such that for every point $a\in A,$ $P(h\in U)>0$ for any neighborhood $U$ of $a,$ and
			the closed span of the elements of $A$ is not equal to $\HH,$ 
			we will show that we can grow $A$ by a single element which not in closed span of the elements of $A$.
			We can only do this a countable number of times because the Hilbert space dimension of $\HH$ is countable.
			(Otherwise, via Gram-Schmidt, we could construct an uncountable orthonormal set by transfinite induction.)

			Choose $g$ such that $g \perp a$ for all $a\in A.$ Now, $P(h \perp g)<1.$ So there must be a point $b$ such that
			$P(h\in U) >0$ for every neighborhood of $b$ and $b$ is not perpendicular to $g,$ therefore, $b$ is not in the span of the elements of $A.$
		\end{proof}

		Suppose $h_1, h_2, \ldots$ is a sequence independent identically distributed of random vectors.
		Let $A$ be as in Lemma 1. Index $A$ a a sequence $a_n.$
		Let $B_{m,n}$ be a ball of radius $1/m$ centered at $a_n$
		Almost surely, the sequence $h_i$ must visit $B_{m,n}$ infinitely often,
		as $P(h_i\in B_{m,n})>0$. Therefore $A$ is a subset of the closure
		of the values of the sequence. (We have essentially the fact that a random function $f:\mathbb{N}\rightarrow \mathbb{N}^2$ is surjective with infinite multiplicity.)
		\end{proof}

\section{Sampling random Nevanlinna-Pick problems}
	Given a committee space $\HH,$ the \dfn{natural domain} of the multiplier algebra of $\HH$ is the set of all tuples of matrices $(x_1, x_2, \ldots) $ such that
	the map taking $M_{z_i} \mapsto x_i$ is a completely contractive homomorphism. In the commutative case, where $\HH$ might also
	be interpreted as a space of complex analytic functions, the most important and familiar points are the scalars, which are, in principle, the domain where it makes
	sense to evaluate all functions in $\MH.$

	Given $X=(x^{(1)}, \ldots, x^{(m)})$ a sequence of points in the natural domain of the space of multipliers, and compatiable target values $y^{(k)}_{ij},$
	the \dfn{Nevanlinna-Pick problem} asks 
	when there are $f_{ij}$ such that $f_{ij}(x^{(k)})=y^{(k)}_{ij}$ and the block multiplier $[M_{f_{ij}}]_{i,j}$ has norm less than or equal to $1.$
	Let $P_X$ be the projection onto $\{h | h(x^{(k)})=0 \textrm{ for all }k\}^{\perp}.$ Note that, as $\{h | h(x^{(k)})=0 \textrm{ for all }k\}$ is an invariant
	subspace for the shifts, its orthogonal complement is invariant for the adjoints, that is,
		$$M_f^*P_X = P_X M_f^* P_X, P_XM_f = P_X M_f P_X.$$
	A neccessary condition for the Nevanlinna-Pick problem to be solvable is for there to be
	functions $g_{ij}(x^{(k)})=y^{(k)}_{ij}$ such that $[P_X M_{g_{ij}} P_X]_{i,j}$ has norm less than $1.$ (In fact, the norm of this block (operator) matrix is independent
	of the choice of $g_{ij},$ when they exist.)
	Moreover, whenever we are working in a complete Nevanlinna-Pick space the condition is also sufficient. See \cite{ampi} and \cite{BMV}
	for a theory of commutative and noncommutative complete Nevanlinna-Pick spaces respectively.
	One way to test for the potential failure of the column-row property in a complete Nevanlinna-Pick space would be to choose a random Nevanlinna-Pick problem
	and then show that when we interpret the target data as a row, the problem is not solvable, but when the target data is interpreted as a column it is.
	Define,
		$$R^{X}_v = E(P_{X} M_{f_v} P_{X} M_{f_v}^*P_{X}), C^X_v = E(P_{X} M_{f_v}^* P_{X}M_{f_v}P_{X}).$$
	We interpret $\|R^{X}_v\|$ and $\|C^X_v\|$ as the minimum norm of a solutions to a Nevanlinna-Pick probem such that $y_i^{k}=f_i(x^{(k)})$ where the
	$f_i$ are random multipliers when $\HH$ is a complete Nevanlinna-Pick space.

	Let $x$ be a random variable taking values in the natural domain of the space of multipliers. We say $x$ is a \dfn{random point} whenever $P(h(x)=0)<1$ for every
	$h \in \HH.$
	\begin{lemma}\label{stronk}
		Let $\HH$ be a committee space.
		Let $x^{(1)}, x^{(2)}, \ldots$ be an infinite sequence of identically distributed independent random points, and let $X_n=(x^{(1)},\ldots, x^{(n)}).$
		The sequence $P_{X_n} \rightarrow I$ in the strong operator topology.
	\end{lemma}
	\begin{proof}
		Note that the sequence $P_{X_n}$ must converge monotonically to some projection
		in the strong operator topology, so it suffices to show that its range is all of $\HH.$		

		Define a random vector $h$ by choosing
		random point $x$ and then choosing a random unit vector in the finite dimensional space $\{\eta | \eta(x^{(k)})=0\}^\perp.$
		Note $h$ is a random vector because $P(\eta(x)=0)<1$ by definition of random point and therefore $P( h \perp g) < 1.$
	
		Note that since the closed span of $h^{(k)}$ is almost surely all of $\HH$ by Lemma \ref{mainlemma}, then we are done as each $h^{(k)}$ is
		in the range of $P_{X_n}$ for all $n\geq k.$
	\end{proof}

	As a corollary of our main result,
	we see that random Nevanlinna-Pick problems satisfy the column-row property.
	\begin{corollary}\label{randommatrix}
		Let $\HH$ be a committee space.
		Let $x^{(1)}, x^{(2)}, \ldots$ be an infinite sequence of identically distributed independent random points, and let $X_n=(x^{(1)},\ldots, x^{(n)}).$
		Let $v$ be a random multiplier model. 
		Then, almost surely,
			$$ \|R_v\|=\lim_{n\rightarrow \infty}  \|R^{X_n}_v\| \leq \lim_{n\rightarrow \infty}  \|C^{X_n}_v\|=\|C_v\|.$$
	\end{corollary}
	\begin{proof}
		By Theorem \ref{mainresult}, it is enough to show the limits converge to the appropriate values.
		Recall that $P_{X_n} \rightarrow I$ in the strong operator topology almost surely by Lemma \ref{stronk}.
		Moreover, note the sequence $P_{X_n}$ is increasing.
		
		Note,
		\begin{align*}
			R^{X_n}_v 	&= E(P_{X_n} M_{f_v} P_{X_n} M_{f_v}^*P_{X_n})\\
					&= E(P_{X_n} M_{f_v} M_{f_v}^*P_{X_n})\\
					&= P_{X_n} E( M_{f_v} M_{f_v}^*) P_{X_n}.
		\end{align*}
		Therefore, $R^{X_n}_v$ converges in the strong operator topology to $R_v,$ and each $\|R^{X_n}_v\|\leq \|R_v\|.$
		
		Now consider,
		\begin{align*}
			C^{X_n}_v 	&= E(P_{X_n} M_{f_v}^* P_{X_n} M_{f_v}P_{X_n})\\
					&= E(M_{f_v}^* P_{X_n} M_{f_v}).
		\end{align*}
		Therefore,
			$$C_v - C^{X_n}_v  = E(M_{f_v}^* (I- P_{X_n}) M_{f_v})$$
		is positive semi-definite.
		So, we see that $\|C^{X_n}_v\|  \leq \|C_v\|.$
		Recall, from the proof of Lemma \ref{collemma},
			$$\|C_v\| = \langle C_v 1, 1 \rangle = \sum \|z^\gamma\|^2 w_\gamma.$$
		So, it is sufficient to show that
			$$\lim_{n\rightarrow \infty} \langle C^{X_n}_v 1, 1\rangle = \sum_\gamma \|z^\gamma\|^2 w_\gamma.$$
		Calculating,
			\begin{align*}
			C^{X_n}_v 	&= E(M_{f_v}^* P_{X_n} M_{f_v}).\\
					&= E(\sum_{\gamma,\gamma'} \overline{v_\gamma} v_{\gamma'}M_{z^{\gamma}}^*P_{X_n}M_{z^{\gamma'}}) \\
					&= \sum_{\gamma} w_\gamma M_{z^{\gamma}}^*P_{X_n}M_{z^{\gamma}}. 
			\end{align*}
		Now consider,
			\begin{align*}
		\langle C^{X_n}_v 1, 1\rangle &= \langle\sum_{\gamma} w_\gamma  M_{z^{\gamma}}^*P_{X_n}M_{z^{\gamma}}1,1\rangle \\
					 &= \sum_{\gamma} w_\gamma  \langle P_{X_n}z^\gamma,P_{X_n} z^\gamma \rangle\\
					 &= \sum_{\gamma}   \|P_X z^\gamma\|^2 w_\gamma.
			\end{align*}
		As $n \rightarrow \infty$ this converges monotonically to $\sum_\gamma \|z^\gamma\|^2 w_\gamma,$ since $P_{X_n}$ is increasing
		and $P_{X_n} \rightarrow I$ in the strong operator topology.
	\end{proof}
	\subsection{Some conclusions on potential experiments}
		Suppose one were looking for a counterexample to the claim some space, for example the Drury-Arveson space, satisfied the true column row property.
		The obvious thing to try is to take a random Nevanlinna-Pick problem. However, our result Theorem \ref{randommatrix} says that if we choose
		many interpolation nodes and a long column, we are doomed. Therefore, it is advisable to choose the least number of interpolation nodes possible and the shortest
		conceivable column length. (For example, $2.$) Furthermore, if we are forced to choose a lot of nodes, it would be best not to choose them uniformly.
		
		 Experiments performed on the two variable Fock space, performed by Augat, Jury, and the present author, which are described in \cite{AJKP},
		gave fairly poor results on random data. For examples arising from a single $2$ by $2$ matrix interpolation node with a column of length $2$, choosing elements at random yielded
		a column-row constant of only about $1.0043.$ At the time, it was thought that ``bigger is better," however our Corollary \ref{randommatrix} says that is not the case.
		Later, hand-crafted examples gave showed that the column-row property fails for any constant.

\bibliography{references}
\bibliographystyle{plain}

\end{document}